\documentclass[12pt,a4paper]{article}

\usepackage[dvips]{color}
\usepackage{color}
\usepackage{xcolor}

\usepackage{amsfonts}
\usepackage{amsfonts}
\usepackage{amsfonts}
\usepackage{amsfonts}
\usepackage{amsfonts}
\usepackage{mathrsfs}
\usepackage{amsfonts}
\usepackage{mathrsfs}
\usepackage{amsfonts}
\usepackage{amsfonts}
\usepackage{mathrsfs}
\usepackage{mathrsfs}
\usepackage{amsfonts}
\usepackage{amsfonts}
\usepackage{amsfonts}
\usepackage{amsfonts}
\usepackage{mathrsfs}

\usepackage{amsfonts}
\usepackage{mathrsfs}
\usepackage{mathrsfs}
\usepackage{mathrsfs}
\usepackage{mathrsfs}
\usepackage{amsfonts}
\usepackage{mathrsfs}
\usepackage{mathrsfs}
\usepackage{amsfonts}
\usepackage{amsfonts}
\usepackage{amsfonts}
\usepackage{amsfonts}
\usepackage{amsfonts}
\usepackage{amsfonts}
\usepackage{amsfonts}
\usepackage{amsfonts}
\usepackage{amsfonts}
\usepackage{mathrsfs}
\usepackage{amsfonts}
\usepackage{mathrsfs}
\usepackage{amsfonts}
\usepackage{mathrsfs}
\usepackage{amsfonts}
\usepackage{amsfonts}
\usepackage{amsfonts}
\usepackage{amsfonts}
\usepackage{amsfonts}
\usepackage{amsfonts}
\usepackage{amsfonts}
\usepackage{amsfonts}
\usepackage{amsfonts}
\usepackage{amsfonts}
\usepackage{amsfonts}
\usepackage{amsfonts}
\usepackage{mathrsfs}
\usepackage{mathrsfs}
\usepackage{mathrsfs}
\usepackage{mathrsfs}
\usepackage{mathrsfs}
\usepackage{mathrsfs}
\usepackage{mathrsfs}
\usepackage{mathrsfs}
\usepackage{mathrsfs}
\usepackage{amsfonts}
\usepackage{amsfonts}
\usepackage{amsfonts}
\usepackage{amsfonts}
\usepackage{amsfonts}
\usepackage{amsfonts}
\usepackage{mathrsfs}
\usepackage{amsfonts}
\usepackage{amsfonts}
\usepackage{amsfonts}
\usepackage{amsfonts}
\usepackage{amsfonts}
\usepackage{amsfonts}
\usepackage{amsfonts}
\usepackage{amsfonts}
\usepackage{amsfonts}
\usepackage{amsfonts}
\usepackage{amsfonts}
\usepackage{amsfonts}
\usepackage{mathrsfs}
\usepackage{amsfonts}
\usepackage{mathrsfs}
\usepackage{amsfonts}
\usepackage{amsfonts}
\usepackage{amsfonts}
\usepackage{mathrsfs}
\usepackage{amsfonts}
\usepackage{amsfonts}
\usepackage{amsfonts}
\usepackage{amsfonts}
\usepackage{amsmath}
\usepackage{mathrsfs}
\usepackage{amsfonts}
\usepackage{amsfonts}
\usepackage{amsfonts}
\usepackage{amsfonts}
\usepackage{amsfonts}
\usepackage{amsfonts}
\usepackage{amsfonts}
\usepackage{amsfonts}
\usepackage{amsfonts}
\usepackage{amsfonts}
\usepackage{amsfonts}
\usepackage{mathrsfs}
\usepackage{amsfonts}
\usepackage{amsfonts}
\usepackage{amsfonts}
\usepackage{amsfonts}
\usepackage{amsmath}
\usepackage{mathrsfs}
\usepackage{mathrsfs}
\usepackage{mathrsfs}
\usepackage{mathrsfs}
\usepackage{mathrsfs}
\usepackage{mathrsfs}
\usepackage{mathrsfs}
\usepackage{mathrsfs}
\usepackage{amsfonts}
\usepackage{amsfonts}
\usepackage{amsfonts}
\usepackage{amsfonts}
\usepackage{amsfonts}
\usepackage{amsfonts}
\usepackage{amsfonts}
\usepackage{amsfonts}
\usepackage{amsfonts}
\usepackage{amsfonts}
\usepackage{amsfonts}
\usepackage{amsfonts}
\usepackage{mathrsfs}
\usepackage{amsfonts}
\usepackage{mathrsfs}
\usepackage{amsfonts}
\usepackage{amsfonts}
\usepackage{amsfonts}
\usepackage{amsfonts}
\usepackage{amsfonts}
\usepackage{amsfonts}
\usepackage{amsfonts}
\usepackage{amsfonts}
\usepackage{amsfonts}
\usepackage{amsfonts}
\usepackage{amsfonts}
\usepackage{amsfonts}
\usepackage{amsfonts}
\usepackage{amsfonts}
\usepackage{amsfonts}
\usepackage{amsfonts}
\usepackage{amsfonts}
\usepackage{amsfonts}
\usepackage{mathrsfs}
\usepackage{mathrsfs}
\usepackage{mathrsfs}
\usepackage{mathrsfs}
\usepackage{amsfonts}
\usepackage{amsfonts}
\usepackage{amsfonts}
\usepackage{amsfonts}
\usepackage{amsfonts}
\usepackage{amsfonts}
\usepackage{amsfonts}
\usepackage{amsfonts}
\usepackage{amsfonts}
\usepackage{amsfonts}
\usepackage{amsfonts}
\usepackage{amsfonts}
\usepackage{amsfonts}
\usepackage{amsfonts}
\usepackage{amsfonts}
\usepackage{amsfonts}
\usepackage{amsfonts}
\usepackage{mathrsfs}
\usepackage{mathrsfs}
\usepackage{amsfonts}
\usepackage{amsfonts}
\usepackage{amsfonts}
\usepackage{amsfonts}
\usepackage{amsfonts}
\usepackage{amsfonts}
\usepackage{amsfonts}
\usepackage{amsfonts}
\usepackage{amsfonts}
\usepackage{amsfonts}
\usepackage{amsfonts}
\usepackage{amsfonts}
\usepackage{amsfonts}
\usepackage{mathrsfs}
\usepackage{mathrsfs}
\usepackage{amsfonts}
\usepackage{amsfonts}
\usepackage{amsfonts}
\usepackage{amsfonts}
\usepackage{amsfonts}
\usepackage{amsfonts}
\usepackage{mathrsfs}
\usepackage{mathrsfs}
\usepackage{amsfonts}
\usepackage{amsfonts}
\usepackage{amsfonts}
\usepackage{amsfonts}
\usepackage{amsfonts}
\usepackage{amsfonts}
\usepackage{amsfonts}
\usepackage{amsfonts}
\usepackage{amsfonts}
\usepackage{amsfonts}
\usepackage{amsfonts}
\usepackage{amsfonts}
\usepackage{amsfonts}
\usepackage{amsfonts}
\usepackage{mathrsfs}
\usepackage{amsfonts}
\usepackage{amsfonts}
\usepackage{amsfonts}
\usepackage{amsfonts}
\usepackage{amsfonts}
\usepackage{amsfonts}
\usepackage{amsfonts}
\usepackage{amsfonts}
\usepackage{amsfonts}
\usepackage{amsfonts}
\usepackage{amsfonts}
\usepackage{amsfonts}
\usepackage{amsfonts}
\usepackage{amsfonts}
\usepackage{amsfonts}
\usepackage{amsfonts}
\usepackage{amsmath}
\usepackage{mathrsfs}
\usepackage{amsfonts}
\usepackage{amsfonts}
\usepackage{amsfonts}
\usepackage{amsfonts}
\usepackage{amsmath}
\usepackage{mathrsfs}
\usepackage{amsfonts}
\usepackage{mathrsfs}
\usepackage{mathrsfs}
\usepackage{mathrsfs}
\usepackage{mathrsfs}
\usepackage{amsmath}
\usepackage{mathrsfs}
\usepackage{amsfonts}
\usepackage{color}
\usepackage{mathrsfs}
\usepackage{mathtools}
\usepackage{stmaryrd}
\usepackage{amsmath}
\usepackage{amssymb}

\usepackage{bbm}
\usepackage{mathrsfs}
\usepackage{graphicx}

\usepackage{enumerate}
\usepackage{thm}

\makeatletter
\def\tank#1{\protected@xdef\@thanks{\@thanks
        \protect\footnotetext[0]{#1}}}
\def\bigfoot{

    \@footnotetext}
\makeatother

\newcommand{\ba}{\begin{eqnarray}}
\newcommand{\ea}{\end{eqnarray}}
\newtheorem{thm}{Theorem}[section]
\newtheorem{prop}{Proposition}[section]

\newtheorem{defi}{Definition}[section]
\newtheorem{rmk}{Remark}[section]

\newtheorem{ass}{Assumption}[section]
\newtheorem{exmp}{Example}[section]
\numberwithin{equation}{section}

{\theorembodyfont{\rmfamily}

\newenvironment{proof}{Proof}{\hfill $\Box$}

\def\RR{\mathbb{R}}
\def\PP{\mathbb{P}}
\def\FF{\mathbb{F}}
\def\EE{\mathbb{E}}
\def\NN{\mathbb{N}}

\def\cA{{\mathcal A}}
\def\cB{{\mathcal B}}

\def\cF{{\mathcal F}}

\def\be{{\beta}}
\def\de{{\delta}}

\def\si{{\sigma}}

\def\De{{\Delta}}

\def\Om{{\Omega}}

\def\al{{\alpha}}

\def\be{{\beta}}

\def\de{{\delta}}
\def\De{{\Delta}}

\def\si{{\sigma}}

\def\eps{{\epsilon}}

\def\th{{\theta}}
\def\EE{\mathbb{ E}}
\def\th{{\theta}}
\def\si{{\sigma}}
\def\al{{\alpha}}

\setcounter{equation}{0}

\allowdisplaybreaks

\begin{document}

\title{\Large \bf Accessibility of SPDEs driven by pure jump noise and
its applications}
\date{}
\author{{Jian Wang}$^1$\footnote{E-mail:wg1995@mail.ustc.edu.cn}~~~{Hao Yang}$^2$\footnote{E-mail:yanghao@hfut.edu.cn}~~~ {Jianliang Zhai}$^1$\footnote{E-mail:zhaijl@ustc.edu.cn}~~~ {Tusheng Zhang}$^{3}$\footnote{E-mail:Tusheng.Zhang@manchester.ac.uk}
\\
 \small  1. School of Mathematical Sciences,
 \small  University of Science and Technology of China,\\
 \small  Hefei, Anhui 230026, China.\\
 \small 2. School of Mathematics, Hefei University of Technology, Hefei, Anhui 230009, China. \\
 \small  3. Department of Mathematics, University of Manchester,\\
 \small  Oxford Road, Manchester, M13 9PL, UK.
}

\maketitle

\begin{center}
\begin{minipage}{130mm}
{\bf Abstract:}
In this paper, we develop a new method to obtain the accessibility of stochastic partial differential equations driven by additive pure jump noise. An important
novelty of this paper is to allow  the driving noises to be degenerate. As an application, for the first time, we obtain the accessibility of a  class of stochastic equations
driven by pure jump degenerate noise, which cover 2D stochastic Navier-Stokes equations, stochastic Burgers
type equations, singular stochastic $p$-Laplace equations,  stochastic fast diffusion equations, etc. As a further application, we establish the ergodicity of singular stochastic $p$-Laplace equations and stochastic fast diffusion equations driven by additive pure jump noise, and we remark that the driving noises could be
L\'evy processes with heavy tails.

\vspace{3mm} {\bf Keywords:}
 Accessibility; pure jump noise; degenerate; ergodicity; stochastic partial differential equation; 2D stochastic Navier-Stokes equation;
 singular stochastic $p$-Laplace equation; stochastic fast diffusion equation

\vspace{3mm} {\bf AMS Subject Classification (2020):}
60H15; 60G51; 37A25; 60H17.
\end{minipage}
\end{center}


\renewcommand\baselinestretch{1.2}
\setlength{\baselineskip}{0.28in}
\section{Introduction and motivation}\label{Intr}

Let $H$ be a topological space with Borel $\sigma$-field $\mathcal{B}(H)$, and let $\mathbb{X}:=\{X^x(t),t\geq0;x\in H\}$ be an $H$-valued  Markov process on some
probability space $(\Omega,\mathcal{F},\PP)$.
$\mathbb{X}$ is said to be  irreducible in $H$ if for each $t>0$ and $x\in H$
\begin{center}
$\PP(X^x(t)\in B)>0$ \quad for any non-empty open set $B$.
\end{center}
$\mathbb{X}$ is said to be accessible to
$x_0\in H$ if the resolvent $R_{\lambda}, \lambda>0$ satisfies
\begin{equation*}
R_{\lambda}(y,U)=\lambda \int_0^{\infty}e^{-\lambda t}\PP(X^x(t)\in U)dt>0
\end{equation*}
for all $x\in H$ and all neighborhoods $U$ of $x_0$, where $\lambda > 0$ is arbitrary. It is clear that irreducibility implies accessibility.

The accessibility and irreducibility are  fundamental properties of stochastic dynamic systems. The importance of the study of the accessibility and irreducibility  lies in  its relevance in the analysis of the ergodicity of  Markov processes. Indeed, accessibility is often used as a replacement of irreducibility when one proves ergodicity; see \cite{Hairer M 2006,Kapica}.

%

Now we mention the main existing results concerning the accessibility of stochastic partial differential equations (SPDEs) driven by  pure jump noise.
The authors in \cite{Xulihu} obtained the accessibility to zero of stochastic real valued Ginzburg-Landau equation on torus $\mathbb{T}=\mathbb{R}\setminus \mathbb{Z}$  driven by cylindrical symmetric $\alpha$-stable process with $\alpha\in(1,2)$.
In \cite{HE}, the authors established the accessibility to zero of stochastic shell models
driven by additive pure jump noise.
Recently, the authors in \cite{DRL} obtained the accessibility to zero of a class of semilinear SPDEs with Lipschitz coefficients driven by multiplicative pure jump noise. The driving noises they considered are a class of subordinated Wiener processes.
There are also several papers to study the irreducibility of SPDEs/SDEs driven by pure jump L\'evy noise; see \cite{AP,DWX2020,FHR 2016,17,WXX,WX,XZ}.
The existing methods on the accessibility  are basically along the same lines as that for the Gaussian case, that is, two ingredients  play very important role: the (approximate) controllability of the associated PDEs and the support of  L\'evy processes/stochastic convolutions  on  path spaces.
In the existing paper on the accessibility  of  SPDEs driven by pure jump noise,  very restrictive assumptions on the driving noises are placed: The driving noises are more or less in the class of stable processes with additional assumptions; see, e.g., Condition (ii)  in \cite{Xulihu}, Assumption (A4) and the assumption that the L\'evy measure $\nu$ of the subordinator satisfies $\nu((0,\infty))=\infty$ in \cite{DRL}, Assumptions 2.3 and 2.8 in \cite{HE}, etc.  
The use of existing methods to deal with the case of other types of pure jump noises appears to be unclear. And using these methods to study the accessibility  of stochastic equations with singular/highly nonlinear terms would be very hard, if not impossible.

In a recent paper \cite{WYZZ}, we developed a new
criterion for the irreducibility of SPDEs driven by multiplicative pure jump noise.  We point out that under the assumptions on the driving noises in the existing paper on the accessibility  mentioned above, one can obtain the irreducibility of the SPDEs as an application of the main results in \cite{WYZZ}.


To the best of our knowledge, there are no results on the accessibility of SPDEs driven by pure  jump degenerate noises. This strongly motivates the current paper.
The aim of this paper is to investigate the accessibility of a class of SPDEs driven by additive (possibly degenerate) jump noise; see Theorem \ref{thm1} in this paper.
For the first time,  we obtain the accessibility of a class of SPDEs driven by additive pure jump degenerate noise, including many stochastic Hydrodynamical systems, such as 2D stochastic Navier-Stokes equations, stochastic Burgers
type equations, etc., and
a class of SPDEs with  weakly dissipative coefficients, such as singular stochastic $p$-Laplace equations,  stochastic fast diffusion equations, etc.; see Proposition \ref{prop 3} in this paper.
As a further application, we establish the ergodicity of a class of SPDEs with  weakly dissipative coefficients, which covers singular stochastic $p$-Laplace equations and  stochastic fast diffusion equations, etc. The driving noises could be
L\'evy processes with heavy tails.
%
%

The organization of this paper is as follows. In Section 2,  we introduce the framework and prove the main results. In Section 3,  we provide applications to SPDEs, including the accessibility of a class of SPDEs and the ergodicity of a class of SPDEs with  weakly dissipative coefficients.



\section{Accessibility}

In this section, we will introduce the framework and prove the main results.

Let $(\Om, \cF, \FF,\PP)$, where $\FF = \{\cF_t\}_{t \geq 0}$, be a filtered probability space satisfying the usual conditions. Let
$
V \subset H \cong H^* \subset V^*
$
be a Gelfand triple.
Let $\nu$  be a  given  $\si$-finite measure on $H$ satisfying  $\nu(\{0\})=0$ and $\int_{H}\|z\|_H^2\wedge1\nu(dz)<\infty$.
Set $N: \cB(H\times\RR^+) \times \Omega\rightarrow \bar{\NN} = \NN \cup \{0, \infty\}$ be the time homogeneous Poisson random measure with intensity measure $\nu$.  $\tilde{N} (dz,dt) = N(dz,dt) - \nu(dz)dt$ denotes the compensated Poisson random measure.
 $$L(t)=\int_0^t\int_{0<\|z\|_H\leq1}z\tilde{N}(dz,ds) + \int_0^t\int_{\|z\|_H>1} zN(dz,ds), t\geq0$$ defines an $H$-valued L\'evy process.
We consider the following SPDEs driven by $L(t)$
\begin{eqnarray}\label{meq1}
  dX(t)&=& \cA (X(t)) dt +dL(t),
\end{eqnarray}
where the mapping $\cA:V\rightarrow V^*$ is a Borel function.

\begin{defi}
An $H$-valued $\rm c\grave{a}dl\grave{a}g$ $\FF$-adapted process $X^x$ is called a solution of (\ref{meq1}) with initial data $x\in H$ if the following conditions are satisfied.

(I) $X^x(t,\omega)\in V$ for $dt\otimes \PP$-almost all $(t,\omega)\in[0,\infty)\times\Omega$, where
$dt$ stands for the Lebesgue measures on $[0,\infty)$;

(II)  $\int_0^t|\cA(X^x(s))|_{V^*}ds<\infty$, $t\geq 0$, $\PP$-a.s.;

(III) As an equation in $V^*$, the following equality holds, $\PP$-a.s.,
\ba\label{eq o 1}
X^x(t) = x + \int_0^t{\cA(X^x(s))ds}+ L(t), \quad t\geq 0.
\ea

\end{defi}

We also need to consider
\begin{eqnarray}
  &&dX^{\eps}(t)= \cA (X^{\eps}(t)) dt +\int_{0<\|z\|_H\leq \eps}z\tilde{N}(dz,dt),\ \epsilon\in (0,1],\label{meq2}\\
  &&dY(t)= \cA (Y(t)) dt.\label{meq3}
\end{eqnarray}

In the following, we use the notation $X^x$, $X^{\eps,x}$ and $Y^x$ to indicate the solutions of (\ref{meq1}), (\ref{meq2}) and (\ref{meq3}) starting from $x$, respectively.
We introduce the following assumption.
\begin{ass}\label{A}
%
%
%
\begin{itemize}
  \item[{\bf (A0)}]  $\nu$ is symmetric, i.e., $\nu(C)=\nu(-C)$, $\forall C\in\mathcal{B}(H)$.

 \item[{\bf (A1)}] For any $x \in H$, there exist unique global solutions $X^x$, $X^{\eps,x}$ and $Y^x$ to (\ref{meq1}), (\ref{meq2}) and (\ref{meq3}), respectively, satisfying
     \begin{itemize}
    \item[(A1-1)] For any $x \in H$, $\lim_{t \rightarrow \infty} \| Y^x(t) \|_H = 0$.

 \item[(A1-2)] For any $t>0$ and $x \in H$, $\lim_{\eps \rightarrow 0} \|X^{\eps,x}(t) - Y^{x}(t) \|_H = 0$ in probability.

  \item[(A1-3)] For any $\eta>0$, there exist $(\zeta,t)=(\zeta(\eta),t(\eta))\in (0,\frac{\eta}{2}]\times (0,\infty)$ such that\\
$ \inf_{y \in B(0,\zeta)}\mathbb{P}(\sup_{s\in[0,t]}\|X^y(s)\|_H\leq \eta)>0$.
Here $B(0,\zeta)=\{h\in H:\|h\|_H<\zeta\}$.
\end{itemize}

\end{itemize}
\end{ass}
%

\vskip 0.2cm
The main result of this paper is as follows:
 \begin{thm}\label{thm1}
 Assume that Assumption \ref{A} holds, then $\{X^x\}_{x\in H}$ is accessible to zero.
 \end{thm}
 \begin{proof}
 By the definition of accessibility, we see that the proof of this theorem will be complete once
we prove that  for any fixed $x\in H$ and $\kappa>0$, there exists $T_0>0$ and $\de >0$  such that
 \begin{equation}\label{acc-1}
\PP(\|X^x(t)\|_H \leq \kappa) >0, \quad \forall t \in [T_0, T_0+\frac{\de}{2}].
\end{equation}
In the remaining part of the proof, we show (\ref{acc-1}).

Now, we fix $x\in H$ and $\kappa>0$. By Condition (A1), we see that
there exist $\beta\in (0,\frac{\kappa}{16}]$ and $\delta>0$ such that
\begin{eqnarray}\label{eq Zhai 01}
  \inf_{y\in B(0,\beta)}\mathbb{P}\big(\sup_{0\leq s\leq  \delta}\|X^y(s)\|_{H}\leq \frac{\kappa}{8}\big)>0,
  \end{eqnarray}
and that there exist $\tilde{T}>0$ and $\tilde{\epsilon}>0$ such that
\begin{eqnarray}\label{eq Zhai 02}
  \mathbb{P}\big(\|X^{\tilde{\epsilon},x}(\tilde{T})\|_{H}<\beta\big)>0.
  \end{eqnarray}

Let $\tau_{\tilde{\epsilon}}^1=\inf\{t\geq 0:\int_0^t\int_{\|z\|_{H}>\tilde{\epsilon}}N(dz,ds)=1\}$. $\tau_{\tilde{\epsilon}}^1$ has the exponential distribution with parameter $\nu(\|z\|_{H}>\tilde{\epsilon})<\infty$, that is,
\begin{eqnarray}\label{eq Zhai 03}
\mathbb{P}(\tau_{\tilde{\epsilon}}^1> s)=e^{-\nu(\|z\|_{H}>\tilde{\epsilon})s},\ \ \ \ \mathbb{P}(\tau_{\tilde{\epsilon}}^1\leq s)=1-e^{-\nu(\|z\|_{H}>\tilde{\epsilon})s}.
\end{eqnarray}
Notice that, for any $y\in H$, $\{X^y(t),t\in[0,\tau_{\tilde{\epsilon}}^1)\}$ is the unique solution to (\ref{meq2}) with the initial data $y$ on $t\in[0,\tau_{\tilde{\epsilon}}^1)$, and that $\sigma\{X^{\tilde{\epsilon},x}(t),t\geq 0\}$ and $\sigma\{\tau_{\tilde{\epsilon}}^1\}$ are independent. Hence, by (\ref{eq Zhai 02}) and (\ref{eq Zhai 03}),
\begin{eqnarray}\label{eq Zhai 04}
  \mathbb{P}\big(\|X^x(\tilde{T})\|_{H}<\beta\big)
  &\geq&
  \mathbb{P}\big(\{\|X^x(\tilde{T})\|_{H}<\beta\}\cap \{\tau_{\tilde{\epsilon}}^1>\tilde{T}\}\big)\nonumber\\
  &=&
  \mathbb{P}\big(\{\|X^{\tilde{\epsilon},x}(\tilde{T})\|_{H}<\beta\}\cap \{\tau_{\tilde{\epsilon}}^1>\tilde{T}\}\big)\nonumber\\
  &=&
   \mathbb{P}\big(\|X^{\tilde{\epsilon},x}(\tilde{T})\|_{H}<\beta\big) \mathbb{P}\big(\tau_{\tilde{\epsilon}}^1>\tilde{T}\big)
  >0.
  \end{eqnarray}
 By the Markov property of $\{X^h\}_{h\in H}$, (\ref{eq Zhai 01}) and (\ref{eq Zhai 04}), we have, for any $t\in[\tilde{T},\tilde{T}+\frac{\delta}{2}]$,
\begin{eqnarray*}
  \mathbb{P}(\|X^x(t)\|_{H}\leq \kappa)
  &=&\mathbb{E}\Big[\mathbb{E}[I_{\{\|X^x(t)\|_{H}\leq \kappa\}}|\mathcal{F}_{\tilde{T}}]\Big]\\
  &\geq&\mathbb{E}\Big[\mathbb{E}[I_{\{\|X^y(t-\tilde{T})\|_{H}\leq \kappa\}}|X^x(\tilde{T})=y]I_{[0,\beta)}(\|X^x(\tilde{T})\|_{H})\Big]>0.
\end{eqnarray*}

The proof is complete.
 \end{proof}

\section{Applications}

In this section, as an application of the main result,  we obtain the accessibility of a class of coercive and local monotone SPDEs driven by additive pure jump noise.

 Let us formulate the  assumptions on the coefficients $\mathcal{A}$. Suppose that there exist constants $\al > 1,\ \be \geq 0,\ \th > 0,\ C > 0,\ \widetilde{C}>0$, $\varpi>0$, $F>0$ and a measurable (bounded on balls)  function $\rho: V \rightarrow [0, +\infty)$ such that the following conditions hold for all $v,\ v_1,\ v_2 \in V$:

(H1)  (Hemicontinuity) The map $s \mapsto _{V^*}\langle \mathcal{A}(v_1 + sv_2), v \rangle_V $ is continuous on $\RR$.

(H2) (Local monotonicity) $2_{V^*}\langle \mathcal{A}(v_1)-\mathcal{A}(v_2), v_1-v_2\rangle_V
\leq (C+\rho(v_2))\|v_1 - v_2\|_H^2.$

(H3) (Coercivity) $2_{V^*}\langle \mathcal{A}(v), v \rangle_V + \th \|v\|_V^{\al} \leq F +C\|v\|_H^2.$

(H4) (Growth) $\|\mathcal{A}(v)\|_{V^*}^{\frac{\al}{\al-1}} \leq (F + C\|v\|_V^{\al})\big(1 + \|v\|_H^{\be}\big).$

(H5) $\rho(v)\leq C(1+\|v\|_V^\alpha)(1+\|v\|_H^{\beta}).$


(H6) $2_{V^*}\langle \mathcal{A}(v), v \rangle_V+\widetilde{C}\|v\|_H^\varpi\leq 0.$


Assume that (H1)-(H5) hold, the well-posedness of (\ref{meq1}), (\ref{meq2}) and (\ref{meq3}) was proved in \cite[Theorem 1.2]{BWJ}, and for any $x\in H$ and $\epsilon>0$,
\begin{eqnarray}
&&X^x,X^{\epsilon,x}\in D([0,\infty);H)\cap L^\alpha_{loc}([0,\infty);V),\ \mathbb{P}\text{-a.s.,}\label{eq Space 1}\\
&& Y^x\in C([0,\infty);H)\cap L^\alpha_{loc}([0,\infty);V).\label{eq Space 2}
\end{eqnarray}
Here we denote by $D([0,\infty);H)$ and  $C([0,\infty);H)$ the spaces of all $\rm c\grave{a}dl\grave{a}g$ paths and continuous paths from $[0,\infty)$ to $H$, respectively.

We consider the accessibility of (\ref{meq1}). The main result is stated as follows.

\begin{prop}\label{prop 3}
Assume that (H1)-(H6) hold and $\nu$ is symmetric, then the solution $\{X^x\}_{x\in H}$ of (\ref{meq1}) is accessible to zero.
\end{prop}

\begin{rmk}
 The equation satisfying (H1)-(H6) not only covers many stochastic Hydrodynamical systems, including 2D stochastic Navier-Stokes equations, stochastic Burgers
type equations, 2D stochastic Magneto-Hydrodynamic equations, 2D stochastic Boussinesq model for the B\'enard convection, 2D stochastic Magnetic B\'ernard problem, 3D stochastic Leray $\al$-Model for Navier-Stokes equations and several stochastic Shell Models of turbulence, etc; but also covers
a class of SPDEs with  weakly dissipative coefficients such as stochastic $p$-Laplace equations, stochastic porous medium equations, stochastic fast diffusion equations, etc. We refer \cite{BHZ,Liu, PR} for more details.
\end{rmk}

{\bf Proof  of Proposition \ref{prop 3}}

We use  Theorem \ref{thm1} to prove this proposition, that is, we  verify Conditions (A0)-(A1).  The proof of  \cite[Proposition 4.1]{WYZZ} implies Condition (A1-3); see (4.18) in \cite{WYZZ}. Hence,  we only need to verify Conditions (A1-1) and (A1-2).

Applying the chain rule and Condition (H6),
\begin{equation}\label{yang4}
d\|Y^x(t)\|_H^2= 2_{V^*}\langle \mathcal{A}(Y^x(t)), Y^x(t) \rangle_V\leq-\widetilde{C}\|Y^x(t)\|_H^\varpi,
\end{equation}
which implies that the map $t\rightarrow\|Y^x(t)\|_H^2$ is decrease. Then by a  contradiction argument, it is easy to see that
$$\lim_{t \rightarrow \infty} \| Y^x(t) \|_H = 0.$$
Hence Condition (A1-1) holds.

By Condition (H2) and applying the It\^o formula,
\begin{eqnarray*}
&&e^{-C\int_0^t{1+\rho(Y^x(s))ds}} \|X^{\eps,x}(t) - Y^{x}(t) \|_H^2 \\
&=&  2\int_0^te^{-C\int_0^s{1+\rho(Y^x(l))dl}}\int_{\|z\|_H \leq \eps} \langle X^{\eps,x}(s-) - Y^{x}(s-), z \rangle \tilde{N}(dz, ds) \\
&&+ \int_0^te^{-C\int_0^s{1+\rho(Y^x(l))dl}}\int_{\|z\|_H \leq \eps} \|z\|_H^2 N(dz, ds) .
\end{eqnarray*}
Using a standard stopping time argument,
\begin{equation}\label{NS-3}
\EE \|X^{\eps,x}(t) - Y^{x}(t) \|_H^2 \leq \int_0^t e^{C\int_s^t1+\rho(Y^x(l))dl}ds\int_{\|z\|_H \leq \eps} \|z\|_H^2 \nu(dz).
\end{equation}
By (\ref{eq Space 2}) and Condition (H5),
it is easy to see that  (\ref{NS-3}) implies Condition (A1-2).

The proof of this proposition is complete.\hskip 8cm $\square$


Now, we apply the accessibility to establish the ergodicity of a class of SPDEs with  weakly dissipative coefficients.

We will need the following additional  assumption on the coefficient $\mathcal{A}$.


{\bf(H7)} (Weakly dissipativity) $2_{V^*}\langle \mathcal{A}(v_1)-\mathcal{A}(v_2), v_1-v_2\rangle_V \leq 0.$

The main result of the ergodicity is given as follows:
%

\vskip 0.3cm
\begin{prop}\label{unique}
Assume that Conditions (H1), (H3), (H4), (H6), and (H7) hold, and $\nu$ is symmetric.
Then there exists at most one invariant measure to Eq. (\ref{meq1}).

Moreover, if the embedding $V\subseteq H$ is compact and there exists $\hat\theta\in((2-\alpha)\vee0,2]$ such that
\begin{equation}\label{yang2}
  \int_{\|z\|_H>1}\|z\|_H^{\hat\theta}\nu(dz)< \infty,
\end{equation}
and the following condition holds:

(H3')
\[
2_{V^*}\langle \mathcal{A}(v), v \rangle_V + \th \|v\|_V^{\al} \leq F +C\|v\|_H^{2-\hat\theta}.
\]
Then there exists a unique invariant measure to Eq. (\ref{meq1}).
\end{prop}

Before we give the proof of the proposition, we would like to present two examples of the driving noises which satisfy the conditions of the proposition.
\begin{rmk}
 We remark that Condition (\ref{yang2}) includes a large class of the so-called L\'evy processes with heavy tails, i.e., $ \int_{\|z\|_H>1}\|z\|_H^{\chi}\nu(dz)= \infty$ holds for some $0<\chi \leq 2$. In
the following, we give two examples.

{\bf Example 1: Cylindrical L\'evy process}

Let $(L(t))_{t\geq0}$ be an infinite
dimensional cylindrical L\'{e}vy process given by
\begin{equation*}
L(t)=\sum_{j\in\mathbb{N}}\beta_jL^j(t)e_j,
\end{equation*}
where $\{(L^j(t))_{t\geq0},j\in \mathbb{N}\}$ is a sequence of independent one dimensional pure jump symmetric L\'{e}vy
processes with the same L\'{e}vy measure $\mu$, $\{\beta_j,j\in\mathbb{N}\}$ is a sequence of real numbers (possibly zero) and $\{e_j,j\in\mathbb{N}\}$ is a sequence of orthonormal basis of Hilbert space H. If for some $\theta\in(0,2]$,
\begin{equation*}
(H_{\theta}):\quad\int_{|x|>1}|x|^{\theta}\mu(dx)+\sum_{j\in\mathbb{N}}|\beta_j|^{\theta}<+\infty,
\end{equation*}
then the intensity measure $\nu$ of $(L(t))_{t\geq0}$ is symmetric and satisfies (\ref{yang2}) with $\hat\theta=\theta$. Notice that if $\mu(dx)=dx/|x|^{1+\alpha}$ with $\alpha\in(0,2)$, then $\int_{|x|>1}|x|^{\theta}\mu(dx)<\infty$ holds with $\theta\in(0,\alpha)$. For this case, $(L(t))_{t\geq0}$
is the so-called
cylindrical $\alpha$-stable process with $\alpha\in(0,2)$.

{\bf Example 2:  Subordinated cylindrical Wiener process with a $\alpha/2$-stable subordinator, $\alpha\in(0,2)$}

 Let $\{W_t,t\geq0\}$ be a $Q$-Wiener process on $H$, $Q\in L(H)$ is nonnegative, symmetric, with finite trace and possibly degenerate, i.e., $Ker Q\neq\{0\}$; here $L(H)$ denotes the set of all bounded linear operators on $H$. For $\alpha\in(0,2)$, let
$S^\alpha_t,~t\geq0$ be an  $\alpha/2$-stable subordinator independent of $\{W_t,t\geq0\}$.
The subordinated cylindrical Wiener motion $L^\alpha_t, t\geq0$  on $H$ is defined by
$$
L^\alpha_t:=W_{S^\alpha_t},~~t\geq0.
$$
Then the intensity measure $\nu$ of $L^\alpha_t,~t\geq0$ is symmetric and satisfies (\ref{yang2}) with $\hat\theta\in(0,\alpha)$.

\end{rmk}

{\bf Proof of Proposition \ref{unique}}

Applying Proposition \ref{prop 3},
$(X^x(t))_{t\geq 0}$ is accessible to zero.  For any $x,y\in H$, by (H7), applying the It\^{o} formula to $\|X^x(t)-X^y(t)\|_H^2$ to get
    \begin{eqnarray*}
\mathbb{E}(\|X^x(t)-X^y(t)\|_H^2)\leq \|x-y\|_H^2, \ \forall t\geq0.
\end{eqnarray*}
Hence $\{X^x\}_{x\in H}$ satisfies the so-called $e$-property; see \cite{Kapica, KPS 2010}. This together with the accessibility implies  the uniqueness of the invariant measure (if it exists); see  Theorem 2 in \cite{Kapica}.

Now, we prove the second part of this proposition. Indeed, we only need to prove the existence of the invariant measure.

The proof is in the spirit of \cite{Dong}. Define a function $f$ on $H$ by
$f(u)=(\|u\|_H^2+1)^{\frac{\hat\theta}{2}}$. Applying the It\^{o} formula gives
 \begin{eqnarray*}
&&f(X^x(t))\nonumber\\
&=&f(x)+\hat\theta\int_0^t\frac{_{V^*}\langle \mathcal{A}(X^x(s)), X^x(s) \rangle_V }{(\|X^x(s)\|_H^2+1)^{1-\frac{\hat\theta}{2}}}ds\nonumber\\
&&+\int_0^t\int_{0<\|z\|_H\leq 1}[\|X^x(s-)+z\|_H^2+1]^{\frac{\hat\theta}{2}}-[\|X^x(s-)\|_H^2+1]^\frac{\hat\theta}{2}\tilde N(ds,dz)\nonumber\\
&&+\int_0^t\int_{\|z\|_H> 1}[\|X^x(s-)+z\|_H^2+1]^{\frac{\hat\theta}{2}} -[\|X^x(s-)\|_H^2+1]^{\frac{\hat\theta}{2}}N(ds,dz)\nonumber\\
&&+\int_0^t\int_{0<\|z\|_H\leq 1}[\|X^x(s)+z\|_H^2+1]^{\frac{\hat\theta}{2}} -[\|X^x(s)\|_H^2+1]^{\frac{\hat\theta}{2}}
-\frac{\hat\theta\langle X^x(s),z\rangle}{(\|X^x(s)\|_H^2+1)^{1-\frac{\hat\theta}{2}}}\nu(dz)ds.
\end{eqnarray*}
Using Condition (H3'), (\ref{yang2}), the following inequality
\begin{equation*}
|f(u)-f(v)|\leq C|(\|u\|_H^2+1)^{\frac{1}{2}}-(\|v\|_H^2+1)^{\frac{1}{2}}|^{\hat\theta}\leq C\|u-v\|_H^{\hat\theta},\quad u,~v\in H,
\end{equation*}
 and the Taylor expansion gives
\begin{eqnarray*}
&&\mathbb{E}f(X^x(t))+\frac{\theta\hat\theta}{2}\mathbb{E}\int_0^t\frac{\|X^x(s)\|_V^{\alpha}}{(\|X^x(s)\|_H^2+1)^{1-\frac{\hat\theta}{2}}}ds\nonumber\\
&\leq&f(x)+\frac{F\hat\theta}{2}\mathbb{E}\int_0^t\frac{1}{(\|X^x(s)\|_H^2+1)^{1-\frac{\hat\theta}{2}}}ds+Ct\nonumber\\
&&+C\int_0^t\int_{\|z\|_H> 1}\|z\|_H^{\hat\theta}\nu(dz)ds+C\int_0^t\int_{0<\|z\|_H\leq 1}\|z\|_H^2\nu(dz)ds\nonumber\\
&\leq&f(x)+C(1+t).
\end{eqnarray*}
A standard stopping time argument and the following formulas have been used to obtain the above inequality.
\begin{equation*}
\nabla f(u)=\frac{\hat\theta u}{(\|u\|^2+1)^{1-\frac{\hat\theta}{2}}},\quad \nabla^2 f(u)=\frac{\hat\theta\Sigma_{i=1}^{\infty}e_i\otimes e_i}{(\|u\|^2+1)^{1-\frac{\hat\theta}{2}}}-\frac{\hat\theta(2-\hat\theta)u\otimes u}{(\|u\|^2+1)^{1-\frac{\hat\theta}{2}}}.
\end{equation*}
Here $\{e_i\}_{i\in\mathbb{N}}$ is an orthonormal basis of $H$.
Hence
\begin{equation}\label{yang3}
\mathbb{E}\int_0^t\frac{\|X^x(s)\|_V^{\alpha}}{(\|X^x(s)\|_H^2+1)^{1-\frac{\hat\theta}{2}}}ds\leq C(f(x) +1+t).
\end{equation}
Since $\al + \hat\theta -2  >0$, by (\ref{yang3}),
\begin{eqnarray*}
\mathbb{E}\int_0^t\|X^x(s)\|_V^{\al + \hat\theta -2 }ds
&\leq&
\mathbb{E}\int_0^t\frac{\|X^x(s)\|_V^{\al + \hat\theta -2 }(\|X^x(s)\|^{2-\hat\theta}+1)}{(\|X^x(s)\|^2+1)^{1-\frac{\hat\theta}{2}}}ds\\
&\leq&
C\mathbb{E}\int_0^t\frac{\|X^x(s)\|_V^\alpha+1}{(\|X^x(s)\|^2+1)^{1-\frac{\hat\theta}{2}}}ds
\leq C(1+f(x)+t).
\end{eqnarray*}
Therefore, by the classical Bogoliubov-Krylov argument (cf.\cite{Da1}), there exists at least one invariant measure of (\ref{meq1}).

The proof of Proposition \ref{unique} is complete.
\hskip 8cm $\square$

In the following, we apply Propositions \ref{unique} to obtain the ergodicity of singular stochastic $p$-Laplace equations and stochastic fast diffusion equations.

\begin{exmp}\label{exmp 00}{\bf Singular  Stochastic $p$-Laplace equation}
  Let $\Lambda$ be an open bounded domain in $\mathbb{R}^d$ with a sufficiently smooth boundary. Consider the following Gelfand triple
\begin{equation*}
V:=W^{1,p}_0(\Lambda)\subseteq H:=L^2(\Lambda)\subseteq (W^{1,p}_0(\Lambda))^*
\end{equation*}
and the stochastic $p$-Laplace equation
\begin{equation}\label{example1}
dX(t)=[div(|\nabla X(t)|^{p-2}\nabla X(t))]dt+dL(t),\quad X(0)=x\in H,
\end{equation}
where $p\in(1,2)$ for $d = 1, 2$ and $p \in [ \frac{2d}{d+2}, 2)$ for $d \geq 3$.
\end{exmp}

Notice that when   $p>\frac{2d}{d+2}$, the embedding $W^{1,p}_0(\Lambda)\subseteq L^2(\Lambda)$ is compact. Therefore, applying Proposition \ref{unique} to (\ref{example1}), if
let $p\in(1,2)$ for $d = 1, 2$ and $p \in [ \frac{2d}{d+2}, 2)$ for $d \geq 3$, and that
$\nu$ is symmetric,
then there exists at most one invariant measure to (\ref{example1}).
Moreover, if  $p\in(1\vee \frac{2d}{d+2},2)$ and (\ref{yang2}) holds, then there exists a unique invariant measure to (\ref{example1}).


\begin{exmp}\label{exmp 00}{\bf Stochastic fast diffusion equation}
 Let $\Lambda$ be an open (possibly unbounded) domain in $\mathbb{R}^d$ with a sufficiently smooth boundary. Consider the following Gelfand triple
\begin{equation*}
V:=L^{r+1}(\Lambda)\subseteq H:=W^{-1,2}(\Lambda)\subseteq (L^{r+1}(\Lambda))^*,
\end{equation*}
and the following stochastic partial differential equation:
\begin{equation}\label{example2}
dX(t)=\De(|X(t)|^{r-1}X_t)dt+dL(t),\quad X(0)=x\in H,
\end{equation}
where $r\in(0,1)$ for $d = 1, 2$ and $r \in [ \frac{d-2}{d+2}, 1)$ for $d \geq 3$.

\end{exmp}

 Notice that when  $\Lambda$ is bounded and $r>\max \{0,\frac{d-2}{d+2}\}$, the embedding $L^{r+1}(\Lambda)\subseteq W^{-1,2}(\Lambda)$ is compact. Therefore, applying Proposition \ref{unique} to Eq. (\ref{example2}), if
let $r\in(0,1)$ for $d = 1, 2$ and $r \in [ \frac{d-2}{d+2}, 1)$ for $d \geq 3$, and that $\nu$ is symmetric,
then there exists at most one invariant measure to (\ref{example2}).
Moreover, if $\Lambda$ is bounded, let $r\in(0\vee \frac{d-2}{d+2},1)$ and (\ref{yang2}) holds, then there exists a unique invariant measure to (\ref{example2}).

To prove the above two examples, according to Proposition \ref{unique}, we need to show Conditions (H1), (H3), (H4), (H6), and (H7) hold. To do so,
we refer to  the proof of  Propositions 3.2 and 3.4 in \cite{Liu}.

 \vskip 0.3cm
 \begin{rmk}
 \begin{itemize}
   \item For the cases of (\ref{example1}) with $p\geq 2$ and  (\ref{example2}) with $r\geq 1$, the following strong dissipativity condition holds: there exist $\alpha\geq2$ and $\delta>0$ such that
\begin{equation}\label{strong dissipativity}
  2{  }_{V^*}\langle \cA(v_1)-\cA(v_2), v_1-v_2\rangle_V
     \le -\delta \|v_1-v_2\|_V^{\alpha}, \ v_1,v_2\in V.
\end{equation}
   The existence and uniqueness of the invariant measure for these cases are much easier than the cases of (\ref{example1}) with $p< 2$ and  (\ref{example2}) with $r< 1$. We refer to \cite{MZ} and \cite{Aronson} for more details.

      \item It seems quite difficult to get Examples 3.1 and 3.2  with other methods due to the lack of strong dissipativity of the equations.

 \end{itemize}
 \end{rmk}

\noindent{\bf Acknowledgement}. This work is partially  supported by NSFC (No. 12131019, 11971456, 11721101).

\end{document}